\documentclass[12pt]{amsproc}
\usepackage{amsmath}
\usepackage{amsthm}
\usepackage{amssymb}
\usepackage{framed}
\usepackage[hidelinks]{hyperref}
\usepackage{xcolor}
\usepackage{graphicx}
\usepackage[a4paper,  margin=2.5cm]{geometry}

\newtheorem{de}{Definition}
\newtheorem{ex}{Example}

\newtheorem{pr}{Proposition}
\newtheorem{lm}{Lemma}
\newtheorem{co}{Corollary}
\newtheorem{te}{Theorem}
\newtheorem{re}{Remark}
\theoremstyle{definition}

\newcommand{\dif}{\mathrm{d}}
\newcommand{\ii}{\mathrm{i}}

\newcommand{\norm}[1]{\Vert #1 \Vert}  
\newcommand{\abs}[1]{\vert #1 \vert}  

\newcommand{\gr}{\mathrm{grad}}

\newcommand{\CC}{\mathbb{C}}

\newcommand{\RR}{\mathbb{R}}   
\newcommand{\ZZ}{\mathbb{Z}} 
\newcommand{\NN}{\mathbb{N}}
\newcommand{\Ss}{\mathbb{S}}

\DeclareMathOperator{\di}{div}
\DeclareMathOperator{\cu}{curl}

\begin{document}

\title[Steady Euler flows on $\Ss^3$ and Sasakian 3-manifolds]{Steady Euler flows on the 3-sphere and other Sasakian 3-manifolds}

\author{Radu Slobodeanu}

\address{Faculty of Physics, \\ University of Bucharest \\ P.O. Box Mg-11, Bucharest-M\u agurele, \\ RO-077125, Romania}
\email{radualexandru.slobodeanu@g.unibuc.ro}

\thanks{The author is grateful to the anonymous referee for many valuable comments that improved the manuscript.}



\maketitle

\medskip

\begin{abstract}
We present new steady Euler solutions on the (round) 3-sphere, that bifurcate from an \textit{ansatz} proposed in \cite{kam}, showing that these previously known solutions are not isolated. We also extend this ansatz to any Sasakian 3-manifold, such as the Heisenberg group and $SL(2, \RR)$. 
\end{abstract}

\section{Introduction}
A steady Euler field on a Riemannian $3$-manifold $(M, g)$ is a tangent vector field $u$ on $M$ that satisfies the \emph{stationary Euler equations}
\begin{equation} \label{eul}
\nabla_u u = - \gr \, p, \qquad \di u = 0
\end{equation}
for some pressure function $p$ on $M$. This notion is the natural extension of the velocity of an incompressible inviscid fluid (of constant density) in equilibrium, in a domain of the euclidean space. For more informations on topological and dynamical aspects related to fluid mechanics, see the monograph \cite{arn} and the introductory paper \cite{seltop}.

An equivalent reformulation of \eqref{eul} is
\begin{equation} \label{eulb}
u \times \cu u = \gr \, b, \qquad \di u = 0,
\end{equation}
where $b=p+\frac{1}{2}\abs{u}^2$ is the \textit{Bernoulli function} and $\cu u = (\ast \dif u^\flat)^\sharp$ the \textit{vorticity field} (here $\ast$ is the Hodge star operator on $(M,g)$). It is well-known that the Bernoulli function $b$ must be constant along a steady Euler flow (i.e. $u(b)=0$). As we can see from \eqref{eulb}, $b$ is even (locally) constant if a divergence-free vector field aligns with its own curl, that is, if $u$ belongs to a special class of steady Euler fields, called \textit{Beltrami fields}. These observations lead to important dynamical consequences for a steady Euler flow, cf. Arnold's structure theorem \cite{arn}, in compact analytic setting: either $b$ is not a constant and (away from a codimension at least 1 critical set) the stream and vortex lines of $u$ are constrained to lie on the
regular level sets $b^{-1}(c)$ all of whose connected components are tori (provided that $M$ is boundaryless), so that $u$ has \textit{laminar behaviour}; or $b$ is a constant and $u$ must be a Beltrami field, which may exhibit \textit{turbulent behaviour} (see \cite{seltop} for more general statements). 

In this paper we are interested in constructing smooth explicit solutions (with non-constant $b$) of the stationary Euler equations on the round 3-sphere $\Ss^3$ and other spaces with similar geometry. We mention that some explicit constructions of Beltrami fields ($b$ constant) on the 3-sphere and the 3-torus and their (contact) topological features  have been recently studied in \cite{enc, pssnew}.

The starting point is the following simple \textit{ansatz} of steady flow on $\Ss^3$, proposed in \cite{kam} (see also \cite{kkpsnew}): 
\begin{equation} \label{kkps}
u=A(\cos^2 s) \xi + B(\cos ^ 2 s) \xi^\prime.
\end{equation}
where $(\cos s \, e^{i\phi_1}, \sin s \, e^{i\phi_2})$ is the "toroidal" parametrisation of $\Ss^3$, 
$\xi = \partial_{\phi_1} + \partial_{\phi_2}$ is the Hopf vector field, and $\xi^\prime = \partial_{\phi_1} - \partial_{\phi_2}$ the anti-Hopf vector field (see Appendix for details).
Any vector field $u$ of the form \eqref{kkps} is a (globally defined) steady solution of the Euler equations on $\Ss^3$, for any choice of (smooth) functions $A$ and $B$, and will be called a \textit{KKPS solution} in what follows. Notice that $\cos^2 s$ is a first integral for \textit{both} fields $\xi$ and $\xi^\prime$, and that \eqref{kkps} is reminiscent of the local description of steady flows \cite[Prop.1.5.]{arn} in terms of coordinates that mimic the action-angle variables. A very important property of the KKPS solutions is that not only the Bernoulli function is conserved along the flow but also the velocity norm: $u(\abs{u}^2)=0$.  Some KKPS vector field solutions are $S$-integrable (i.e. they span the 1-dimensional fibres of a globally defined map which is submersion on a dense subset), cf. \cite{slo}.  See \cite{sslo} for an $S$-integrable steady Euler solution that do not belong to the KKPS family. 
\medskip

Two natural questions regarding the class of KKPS solutions arise:

\begin{quote}
\noindent \textbf{1.} Can this class of solutions be extended either by replacing $\xi$ and $\xi'$ with other Killing fields (or with other curl-eigenvectors), and does this construction have an analogue in other spaces with similar geometry (e.g. contact metric manifolds, space forms)?

\noindent \textbf{2.} Are these solutions isolated in the space of solutions on $\Ss^3$?
\end{quote}

To have a brief overview of the first question, let us consider a vector field $u = f_1 K_1 + f_2 K_2$ on $(M,g)$, where $K_1, K_2$ are Killing fields of constant norm (supposed to exist on $M$) and $f_1$, $f_2$ are first integrals of $K_1$ and $K_2$ respectively (i.e. $K_i(f_i)=0$). A quick check shows us that:
$$
\nabla_u u = - f_1 f_2 \, \gr (g(K_1, K_2)) + K_1(f_1 f_2)K_2 + K_2(f_1 f_2)K_1, \quad \di u = 0.
$$
This shows that some cumbersome conditions should be further imposed in order to ensure that $u$ is a steady Euler field (e.g. $f_1$, $f_2$ common first integrals for $K_1$ and $K_2$ and $g(K_1, K_2)=\mathcal{F}(f_1 f_2)$).

In special contexts, it is nevertheless easy to pinpoint the analogue of the KKPS ansatz. For instance, in the negative constant curvature space $\mathbb{H}^3$, seen via the half-space model $\{(x,y,z) : z > 0 \}$,  the following is a steady Euler field, for any choice of $A$ and $B$,
\begin{equation} \label{kkpsHyp}
u=A(z) \partial_x + B(z) \partial_y. 
\end{equation}

The second question (that may be reformulated as: \textit{are all
steady states in the vicinity of a KKPS steady solution also KKPS solutions}?) is a natural and relevant one especially in view of a recent unstability result \cite{kkpsnew} claiming that a certain subclass of KKPS flows in $\Ss^3$, called \textit{shear flows}, are "unstable" in the following sense: the trajectories of the Euler dynamical system starting at certain arbitrarily small perturbations of a shear flow $u$ do not converge to any steady Euler solution contained in a small vicinity of $u$.

In order to answer this second question, we should be able to find out whether the solutions \eqref{kkps} admit (1-parameter) continuous deformations (not leaving the space of solutions) such that there exists no isometry relating the deformed solutions with a KKPS solution. In this respect the following notion will be very useful.
\begin{de}[\cite{const}]
A steady Euler solution $u$ on $(M,g)$ is called \emph{localizable} if its norm is conserved along the flow: $u(\abs{u}^2)=0$ everywhere on $M$.
\end{de}

We can use the localizability property to distinguish between solutions, since we have the following

\begin{lm} \label{conservnorm}
Localizability property is invariant under isometries.
\end{lm}

\begin{proof}
Let $(M, g)$ be a Riemannian manifold and $u$ and $v$ be two tangent vector fields on $M$, related by an isometry $\varphi \in \mathrm{Iso}(M)$: $\dif \varphi (u) = v\circ \varphi$. Then $\dif \varphi (\nabla_u u) = \nabla_v v$, and we have $g\big(\nabla_u u, u \big) = g\big( \dif \varphi (\nabla_u u), \dif \varphi(u)\big)= g\big(\nabla_v v, v \big)$, that is $u(\abs{u}^2)=v(\abs{v}^2)$, from which the conclusion follows. 
\end{proof}

The organization of the paper is as follows. In Section 2 we propose a new ansatz \eqref{rs} that extends the KKPS ansatz \eqref{kkps} to any Sasakian 3-manifold, we give sufficient conditions that it has to satisfy in order to yield steady Euler solutions and we provide some examples, thus partially answering the first Question above. In Section 3 we answer in the negative the second Question, on the 3-sphere, by proving the non-isolation of (some) KKPS solutions; this constitutes our main result:

\begin{te}\label{maint}
The family of KKPS solutions on $\Ss^3$ is not isolated: there exist steady Euler solutions (with non-constant Bernoulli function $b$) as well as Beltrami fields ($b$  constant) that are arbitrarily $C^k$-close to the KKPS family, without being members of it.
\end{te}

The proof is based on the explicit construction of new steady Euler fields that turn out to be deformations of specific KKPS solutions. We end with an Appendix including the main notations and basic facts about Sasakian and 3-sphere geometry as well as a list of steady Euler solutions isometrically related to the KKPS class. We mention that, at various points along the paper, some explicit computations are needed and they have been done using Mathematica \cite{math}; the corresponding files are available from the author, upon request.

\section{Sasakian KKPS ansatz}

In this section we show how the ansatz \eqref{kkps} extends to any Sasakian 3-manifold. Let us briefly recall some basic facts about Sasakian geometry; for more information we refer the reader to the monograph \cite{BGalbook}.

A Sasakian manifold, denoted as $(M, \xi, \eta, \phi, g)$, is a 3-manifold $M$ (or odd-dimensional manifold in general) endowed with the following structures: a contact form $\eta$ (i.e. a 1-form such that $\eta \wedge \dif \eta \neq 0$), a Reeb vector field (uniquely defined by the conditions $\eta(\xi)=1$ and $\dif \eta (\xi, \cdot)=0$), a $(1, 1)$-tensor field $\phi$ such that $\phi^2 = -I + \eta \otimes \xi$ and a Riemannian metric $g$ such that $\tfrac{1}{2}\dif \eta(\cdot, \cdot) = g(\cdot, \phi(\cdot))$,  $\eta(\cdot) = g(\cdot, \xi)$  and with respect to which $\xi$ is a Killing vector field.
We recall that in this context $\phi$ can be seen as a transverse (to $\xi$) complex structure, and that $\phi X = -\nabla_{X} \xi$ for any vector $X$. On each Sasakian 3-manifold one can (locally) choose an adapted orthonormal frame $\{\xi, X_1, X_2=-\phi X_1\}$, whose properties are listed in \cite{PSS} (see  Appendix). 

\medskip
\noindent Any compact Sasakian 3-manifold fits into one of the following diffeomorphism classes~\cite{gei}:
\begin{equation*} \label{3types}
\Ss^3/\Gamma, \qquad  \mathrm{Nil}^3/\Gamma, \qquad \widetilde{SL}(2, \RR)/\Gamma\,,
\end{equation*}
where $\Gamma$ is any discrete subgroup of the isometry group of the corresponding canonical metric. Moreover, Belgun's metric classification ~\cite{bel1, bel2} implies that any possible Sasakian structure is a deformation of type I or of type II of a standard Sasakian structure on each of these spaces~\cite{BGalbook}. These standard structures will be used for exemplification in the last part of this section.

Inspired by \cite{gav}, we propose the following \textit{ansatz} on a Sasakian 3-manifold $M$:
\begin{equation} \label{rs}
u=F(\psi) \xi + \phi (\gr \, G(\psi)), 
\end{equation}
\noindent where $F,G$ are arbitrary smooth real functions, and $\psi$ is a function on $M$. Recall that in \cite{gav}, Gavrilov constructed a compactly supported steady Euler field having the following form (up to a nonconstant factor): $\tilde u=F(\psi)\xi+ \xi \times \gr \psi$ (see \cite[(18)]{gav}), where $\xi$ was a Killing field along which the function $\psi$ is constant. Notice that our term $\phi (\gr \, \psi)$ is a reinterpretation of $\xi \times \gr \psi$.

\medskip
The following proposition tells us how to decide whether a solenoidal vector field $u$ on $M$ is of the form \eqref{rs} with $\psi$ a constant function along the flow of the Reeb field (i.e. $\xi(\psi)=0$). This is particularly useful when there exists a globally adapted orthonormal frame (as it is the case in the examples considered in this paper). We shall also see that, if $H^1(M)=0$, the ansatz \eqref{rs} covers all solenoidal vector fields that satisfy $u(\eta(u))=0$ and commute with $\xi$, i.e. $[\xi, u] = 0$. When $[\xi, u] = 0$, one uses the terminology: "$u$ is $\xi$-invariant" or "$\xi$ is a symmetry of $u$".

We start with an instrumental result.

\begin{lm}\label{instrument}
Let $(M, \xi, \eta, \phi, g)$ be a Sasakian $3$-manifold and $u = f \xi + f_1 X_1 +f_2 X_2$ a vector field on $M$.

$(i)$ $[\xi, u] = 0$ is equivalent to the following set of conditions:
\begin{equation}\label{S1invar}
\xi(f)=0, \quad \xi(f_1)=-(C_0+1)f_2, \quad \xi(f_2)=(C_0+1)f_1
\end{equation}

$(ii)$ $\phi u$ is a closed vector field (i.e. its dual 1-form is closed) if and only if the following conditions are satisfied:
\begin{equation*}\label{closed}
\di u - \xi(f)=0, \quad \xi(f_1)=-(C_0+1)f_2, \quad \xi(f_2)=(C_0+1)f_1.
\end{equation*}
\end{lm}

\begin{proof} 
$(i)$ is a straightforward application of \eqref{struct1}.

\noindent $(ii)$ $\phi u$ is a closed vector field if, by definition, $g(\nabla_Y \phi u, Z) =  g(\nabla_Z \phi u, Y)$, for any vector fields $Y$ and $Z$. After simple computations using \eqref{conne} and \eqref{div123}, we see that this condition is satisfied for $Y=\xi$, $Z=X_1$ iff $\xi(f_2)=(C_0+1)f_1$,  for $Y=\xi$, $Z=X_2$ iff $\xi(f_1)=-(C_0+1)f_1$ and for $Y=X_1$, $Z=X_2$ iff $\di u - \xi(f)=0$. By bilinearity the conclusion follows.
\end{proof}

\begin{pr} 
Let $(M, \xi, \eta, \phi, g)$ be a Sasakian $3$-manifold and $u$ a vector field on $M$. Let $f$, $f_1$, and $f_2$ be the components of $u$ with respect to a (local) adapted orthonormal frame $\{\xi, X_1, X_2=-\phi X_1\}$. If $u$ is of the form \eqref{rs}, for some function $\psi$ on $M$ with the property $\xi(\psi)=0$, then $\xi$ is a symmetry of $u$, i.e. $[\xi, u] = 0$.

If in addition $H^1(M)=0$, the following converse holds: given a divergence-free vector field $u$ satisfying $u(f)=0$, if $\xi$ is a symmetry of $u$, then $u$ is of the form \eqref{rs}, for some function $\psi$ such that $\xi(\psi)=0$.
\end{pr}

\begin{proof}
Let $u$ be a vector field is of the form \eqref{rs} for some function $\psi$ such that $\xi(\psi)=0$. Then $u$ is divergence-free ($\di u =0$ is implied by $\xi(\psi)=0$ as we will show in the next proposition). Applying $\phi$ to both members of \eqref{rs} and using the fact that $\gr \, G(\psi)$ is orthogonal to $\xi$, we obtain that $\phi u$ is a gradient vector field, so in particular is closed. Using first $(ii)$, then $(i)$ of Lemma \ref{instrument}, we obtain that $\xi$ is a symmetry of $u$.

Assume now $H^1(M)=0$, and let $u$ be a divergence-free vector field $u$ satisfying $[\xi, u]=0$. By Lemma  \ref{instrument}, $\phi u$ must be a closed vector field. As $H^1(M)=0$, there exists a function $G$ such that $\phi u = -\gr G$, and in particular $\xi(G)=0$. We deduce that $u$ is of the form $u=f\xi + \phi \gr G$, where $f=\eta(u)$. Since by hypothesis we also have $u(f)=0$, and $\xi(f)=0$ cf.  \eqref{S1invar} and from $\phi u = -\gr G$ we get $X_1(G)=-f_2$ and $X_2(G)=f_1$, we can deduce that $X_1(f)X_2(G)-X_2(f)X_1(G)=0$. As $\xi(f)=0$ and $\xi(G)=0$, we also have $\xi(f)X_1(G)-X_1(f)\xi(G)=0$ and $\xi(f)X_2(G)-X_2(f)\xi(G)=0$. The latter 3 equations shows that the gradients of $f$ and $G$ must be collinear, so $u$ is of the form \eqref{rs} as stated.
\end{proof}

\noindent The following result displays a collection of sufficient conditions that a function $\psi$ should satisfy so that $u$ defined by \eqref{rs} yields a steady Euler solution on $M$.

\begin{pr}\label{myansatzprop}
 Let $(M, \xi, \eta, \phi, g)$ be a Sasakian 3-manifold and $u$ a vector field on $M$ be given by \eqref{rs}, with $\psi$ not identically constant. 

\medskip
\noindent $(i)$ The vector field $u$ is divergence-free if and only if either $F'=2G'$ or $\psi$ is a first integral of the Reeb field $\xi$ (i.e. $\xi(\psi)=0$).

\medskip
\noindent $(ii)$  If $G'$ is nowhere vanishing and the following conditions are satisfied 
\begin{equation}\label{beltcond}
\xi(\psi)=0; \qquad \Delta G(\psi) = F(\psi) \left(\frac{F'(\psi)}{G'(\psi)}-2\right),
\end{equation}
then $u$ is a Beltrami field with proportionality factor $F'(\psi)/G'(\psi)$. 

\noindent In particular, on $M=\Ss^3$ with the standard metric, the vector field $u= k G(\psi) \xi + \phi (\gr \, G(\psi))$ is a strong Beltrami field (curl eigenvector) for the eigenvalue $k\in \ZZ$ if $\xi(\psi)=0$ and $G(\psi)$ is an eigenfunction of the Laplacian (on $\Ss^3$) with eigenvalue $k(k-2)$.

\medskip
\noindent $(iii)$  If the functions $\psi$ and $G$ satisfy the conditions 
\begin{equation}\label{cond}
\xi(\psi)=0, \qquad \Delta G(\psi) = \mathcal{G}(\psi),
\end{equation}
for some (smooth) real function $\mathcal{G}$,  then $u$ is a steady solution of the Euler equations on $M$ with pressure $p=-\tfrac{1}{2}\abs{\gr G(\psi)}^2 - \int_0^\psi (\mathcal{G}(q)+2F(q))G'(q)dq$.
\end{pr}

\begin{proof} 

\noindent $(i)$ Consider a (local) adapted orthonormal frame $\{\xi, X_1, X_2=-\phi X_1\}$, having the properties described in \cite[$\S 2$]{PSS}. With respect to this frame we have 
\begin{equation}\label{ansa}
u = F(\psi) \xi + X_2(G(\psi)) X_1 - X_1(G(\psi))X_2.
\end{equation}
Using \eqref{div123} we have $\di u = \xi(F)+X_1(X_2(G))-X_2(X_1(G))-C_1X_1(G)-C_2 X_2(G)=\xi(F)-2\xi(G)$, where we wrote simply $F$ and $G$ instead of $F(\psi)$ and $G(\psi)$, in order to alleviate the notations. The conclusion is immediate.

\medskip
\noindent $(ii)$ Using \eqref{curl123}, we easily obtain:
\begin{equation}\label{curlan}
\cu u= \big(\Delta G(\psi) +2F(\psi)\big)\xi + X_2(F(\psi))X_1 - X_1(F(\psi))X_2.
\end{equation}
Comparing with \eqref{ansa}, we see that the condition $\cu u = \theta u$ ($u$ is a Beltrami field) is satisfied provided that \eqref{beltcond} holds true.

\medskip
\noindent $(iii)$ In order to prove that $u$ is a steady Euler field, we compute $\nabla_u u$ and show that it is the gradient of some function which represents the pressure of the fluid (up to a constant). We shall use repeatedly \eqref{conne} and the hypothesis that $\psi$ is a first integral of $\xi$ (so $\xi(F(\psi))=\xi(G(\psi))=0$) which implies, via the structure equations, that $\xi(X_1(\psi))=-(C_0 + 1)X_2(\psi)$ and $\xi(X_2(\psi))=(C_0 + 1)X_1(\psi)$, and therefore also $\xi(\abs{\gr \psi}^2)=0$.  After a long but elementary computation, we obtain:
\begin{equation}\label{eulergrad}
\begin{split}
\nabla_u u &= \big(X_2(G(\psi))X_1(F(\psi))-X_1(G(\psi))X_2(F(\psi))\big)\xi +\\
&\big( 2F(\psi)X_1(G(\psi))+\tfrac{1}{2}G'(\psi)^2 X_1\left(X_1(\psi)^2 + X_2(\psi)^2\right)+ G'(\psi)^2 X_1(\psi)\Delta \psi\big)X_1 +\\
&\big( 2F(\psi)X_2(G(\psi))+\tfrac{1}{2}G'(\psi)^2 X_2\left(X_1(\psi)^2 + X_2(\psi)^2\right) + G'(\psi)^2 X_2(\psi)\Delta \psi\big)X_2\\
= & 2F(\psi)G'(\psi)\gr \psi +\tfrac{G'(\psi)^2}{2}\gr (\abs{\gr \psi}^2) +G'(\psi)^2 (\Delta \psi) \, \gr \psi \\
= & 2F(\psi)G'(\psi)\gr \psi +\tfrac{1}{2}\gr (\abs{\gr G(\psi)}^2) + \Delta G(\psi)  \, \gr G(\psi).
\end{split}
\end{equation}
where in the second equality we used again $\xi(\abs{\gr \psi}^2)=0$. Taking $\Delta G(\psi) = \mathcal{G}(\psi)$ into account, it becomes clear that the right hand side of \eqref{eulergrad} is a gradient vector.

\noindent An alternative proof, in the case when $H^1(M)=0$, can be obtained by checking that $[u, \cu u]=0$ which implies \eqref{eulb}; for this we use \eqref{curlan}.
\end{proof}

\begin{re}[Regular case]
A Sasakian structure is \emph{regular} if the Reeb field is the (unit) vertical field of a circle bundle over a surfuce $\Sigma$. In this case, the condition $\xi(\psi)=0$ says that $\psi$ is a \emph{basic} function (it is the composition of a function on the base with the bundle projection $\pi:M\to \Sigma$). In particular, on the $3$-sphere, item $(ii)$ in the above Proposition shows that one can construct strong Beltrami fields (i.e. eigenfields of $\cu$) on $\Ss^3$ purely in terms of eigenfunctions of the Laplacian on $\Ss^2(\frac{1}{2})$, using the Hopf fibration seen as Riemannian submersion $\Ss^3\to\Ss^2(\frac{1}{2})$, cf. also \cite{kom}.
\end{re}

The first examples of functions satisfying \eqref{cond} are provided by the isoparametric functions that are constant along the Reeb flow. Recall that a smooth function $f: M\to \RR$ is called \textit{isoparametric} \cite{car, wa} if
\begin{equation}\label{isodef}
\abs{\gr \psi}^2 =\mathcal{F}_1(\psi), \qquad \Delta \psi = \mathcal{F}_2(\psi).
\end{equation}
On a space form $M$, these conditions characterize a function whose regular level sets form a parallel family of hypersurfaces with constant mean curvature; they also have  constant principal curvatures, i.e. they are  \emph{isoparametric hypersurfaces} (see \cite[Ch.2]{baird}, and for a broader historical survey, \cite{thor}). 

\begin{co}\label{coriso}
If $\psi$ is a isoparametric function with the property $\xi(\psi)=0$,
then $u$ given by \eqref{rs} is a \emph{localizable} steady Euler field on $M$, for any (smooth) functions $F$, $G$.
\end{co}

\begin{proof}
Let $\psi$ be a isoparametric function with the property $\xi(\psi)=0$. The generally valid identity $\Delta G(\psi)=G'(\psi)\Delta \psi - G''(\psi) \abs{\gr \psi}^2$ yields in our case: $\Delta G(\psi)=G'(\psi)\mathcal{F}_2(\psi) - G''(\psi) \mathcal{F}_1(\psi)$, so \eqref{cond} is satisfied and, according to Proposition \ref{myansatzprop}, $u$ is a steady Euler field.

\noindent Observe that $\gr \abs{u}^2= \gr \left(F(\psi)^2 + \abs{\gr \, G(\psi)}^2)\right) =$ $\gr \left(F(\psi)^2 + G'(\psi)^2  \mathcal{F}_1(\psi)\right)$ is collinear with $\gr \psi$ and therefore orthogonal to $u=F(\psi) \xi + G'(\psi) \phi (\gr \, \psi)$ (recall that $g(\xi, \gr \psi)=0$ and $g(X, \phi X)=0$, for any $X$). So $u(\abs{u}^2)=0$ and $u$ is localizable.
\end{proof}

In the remaining part of this section we will illustrate the above results on the three regular Sasakian models, with spherical, Nil, and $\widetilde{SL}_2$-type geometries. It would be interesting to know whether there exist steady Euler fields of the form \eqref{rs} on any \textit{compact} Sasakian 3-manifold, that is on compact quotients of the three standard spaces, endowed with a metric given as a deformation of type I or II of the standard one.

\begin{ex}[Spherical geometry] \label{exkkpsas}
The KKPS ansatz \eqref{kkps} is a particular instance of the new ansatz \eqref{rs}. To see this, consider the isoparametric function $\psi(\cos s \, e^{i\phi_1}, \sin s \, e^{i\phi_2}) = \cos^2 s$ on $\Ss^3$ and \eqref{rs} becomes: $u=[F(\cos^2 s)+\cos 2s G'(\cos^2 s)]\xi-G'(\cos^2 s)\xi^\prime$, that clearly has the form \eqref{kkps}. In particular, the nonvanishing KKPS-type Beltrami fields studied in \cite{pssnew} are of the form \eqref{rs} and satisfy \eqref{beltcond}.
\end{ex}

It is important to notice that the above example is essentially the unique example for Corollary \ref{coriso} on $\Ss^3$. Indeed, the isoparametric functions on a (round) sphere $\Ss^m$ are given by the so called Cartan-M\"unzer polynomials on $\RR^{m+1}$ that can have degree $p=1, 2, 3, 4$ or $6$ \cite{mun, mun2}, degree that coincides with the number of distinct principal curvatures of the regular level sets (which are surfaces in our case $m=3$). As the case $p=1$ is excluded (it cannot be  $\xi$-invariant), the only remaining case is $p=2$ that corresponds to Example \ref{exkkpsas} above. At a first sight this might be surprising, since the following degree 4 polynomial introduced by Nomizu in \cite{nomi} defines an isoparametric function on $\Ss^3\subset \RR^2 \times \RR^2$:
\begin{equation} \label{nomizudef}
\psi(x_1, y_1, x_2, y_2)=\left(\abs{(x_1, x_2)}^2 - \abs{(y_1, y_2)}^2\right)^2 +4\left\langle(x_1, x_2), (y_1, y_2)\right\rangle^2.
\end{equation}

% case $p=3$ is excluded since in this case the multiplicitees of the 3 principal curvatures of the regular level sets should be all equal, which is not allowed by the dimension 3 of interest; case $p=6$ is again excluded since half of the 6 multplicities should be equal.
% and in \textit{Nomizu coordinates}, $\psi\left(e^{\ii \theta}\big(\cos t (\cos \mu, \sin \mu)+ \ii \sin t (\sin \mu, -\cos \mu)\big)\right)=\cos^2 (2t)$.

Equivalently, in Hopf coordinates, the above function is given by $\psi (\cos s \, e^{i\phi_1}, \sin s \, e^{i\phi_2})= \tfrac{1}{4} (3 + \cos(4 s) + 2 \sin^2(2 s) \cos 2 (\phi_1 - \phi_2) )$. Moreover, $\psi$ is constant along $\xi$, so, by Corollary \ref{coriso}, the corresponding vector field $u$ defined by \eqref{rs} is a steady Euler field on $\Ss^3$.

But actually this is not a new solution, as the global isometry defined by $\Psi(x_1, y_1, x_2, y_2) = \left(\frac{1}{\sqrt{2}}(x_1 + y_2), \frac{1}{\sqrt{2}}(y_1 - x_2), \frac{1}{\sqrt{2}}(y_1 + x_2), \frac{1}{\sqrt{2}}(-x_1 + y_2)\right)$ pulls the function $\psi$ in \eqref{nomizudef} back to the function $4(x_1^2+y_1^2)(x_2^2+y_2^2)$, that is $\sin^2 2s$, which is covered by Example \ref{exkkpsas}.

Nonetheless, we shall see in the next section that Nomizu's function is useful for deriving new solutions.

\begin{ex}[Nil geometry]
The 3-dimensional Heisenberg group
$\mathrm{Nil}^3$ is the Lie group of all real matrices
 $\begin{pmatrix} 1 & x & z \\ 0 & 1 & y \\ 0 & 0 & 1 \end{pmatrix}$ on which the subgroup with integer entries $\Gamma=\ZZ^3$ acts by left multiplication. Then the (compact) $3$-dimensional \emph{Heisenberg nilmanifold} is defined as $M=\mathrm{Nil}^3/\Gamma$. The field $\xi=\frac{1}{2\pi}\partial_z$ is the Reeb field associated to the contact form $\eta = 2\pi(\dif z-x \dif y)$ that admits the adapted Sasakian metric $g=\pi(\dif x^2+ \dif y^2)+4\pi^2(\dif z-x \dif y)^2$.

The function $\psi:M \to \RR$, $\psi(x,y,z)= \cos 2 \pi (x - y)$ satisfies \eqref{isodef}, so any vector field of the form \eqref{rs} is a localizable steady Euler solution on $(\mathrm{Nil}^3/\Gamma, \, g)$.
\end{ex}

\begin{ex}[$\widetilde{SL}_2$ geometry] 
On the Lie group $SL(2, \RR)$ (the covering of the unit tangent bundle $T_1 \Sigma_g$ of a compact oriented surface $\Sigma_g$ of genus $g > 1$, and with constant $-2$ Gaussian curvature) we consider the  contact form  $\eta = 2\dif \theta + \frac{1}{y} \dif x$, with (regular) Reeb field $\xi=\frac{1}{2}\partial_\theta$ and the adapted Sasakian metric $g=\frac{1}{2y^2}(\dif x^2+ \dif y^2)+(2\dif \theta + \frac{1}{y} \dif x)^2$, where the coordinates $(x,y,\theta)\in \RR\times (0, \infty) \times [0, 2\pi]$ are given by the Iwasawa decomposition of matrices on $SL(2, \RR)$.

The function $\psi: SL(2, \RR) \to \RR$, $\psi(x,y,z)= x/y$ satisfies \eqref{isodef}, so any vector field of the form \eqref{rs} is a localizable steady Euler solution on $(SL(2, \RR), \, g)$. 

\noindent For other examples and further discussions  in this geometric context, see \cite[Ch.V $\S$4.D]{arn}.
\end{ex}

\section{Solutions on the 3-sphere and non-isolation}

In this section we find new solutions of the steady Euler equations on $\Ss^3$ endowed with the round metric. These solutions form two 1-parameter families of solutions that "bifurcate" from the KKPS class of solutions. This will allow us to prove the main result in Theorem \ref{maint} stating that in general a KKPS solution is not isolated. 

The first family of new solutions belongs to the new ansatz \eqref{rs} based on a deformation of Nomizu's function \eqref{nomizudef}.

\begin{pr} \label{solnomi}
Consider the following family of vector fields on $\Ss^3$, indexed by $a\geq 0$: 
\begin{equation} \label{solnomizu}
u_a = F(\psi_a) \xi + \phi (\gr \, \psi_a),
\end{equation}
where $F$ is a smooth real function, and the function $\psi_a$ is defined in terms of Hopf coordinates on $\Ss^3$ as $\psi_a = \tfrac{1}{4} (3 + \cos(4 s) + 2 a \sin^2(2 s) \cos 2 (\phi_1 - \phi_2))$. The following statements hold true.

\noindent $(i)$ For any $a$, the vector field $u_a$ is a steady Euler solution on $\Ss^3$ with Bernoulli function 
\begin{equation} \label{bernomizu}
b=\tfrac{1}{2}F(\psi_a)^2 -12\psi_a^2 + 16 \psi_a -2 \int_0^{\psi_a} F(q) dq.
	\end{equation}
In particular, if $F(x)=2(3x -2)$, then $u_a$ is a strong Beltrami field, that is, for any $a$, $u_a$ is an eigenvector of the $\cu$ operator for the eigenvalue $\mu=6$.

\noindent $(ii)$ If $a=0$, then $u_a$ is a KKPS steady Euler solution (i.e. of the form \eqref{kkps}). If $a=1$, then $u_a$ is a KKPS steady Euler solution up to isometries.

\noindent $(iii)$ If $a \neq 0,1$, then there exists no isometry relating $u_a$ to any member of the KKPS family \eqref{kkps}.
\end{pr}

\begin{proof} $(i)$  If $a=0$, $\psi_a$ is a function of $\cos^2 s$ that is isoparametric, while if $a =1$, $\psi_a$ is Nomizu's function \eqref{nomizudef}, so it is again isoparametric. If $a \notin \{0, 1\}$, the function $\psi_a$ is not isoparametric, but  we still have $\xi(\psi_a)=0$ (so $\xi (\abs{\gr \psi_a}^2)=0$), and $\Delta \psi_a =  8 (3 \psi_a - 2)$. Since $u_a$ is of the form \eqref{rs} with $G(x)=x$, for any value of $a$ it follows by Proposition \ref{myansatzprop} and Corollary \ref{coriso} that $u$ is a steady Euler field. 

\medskip
\noindent $(ii)$ The field $u_0$ is of KKPS type cf. Example \ref{exkkpsas}. The fact that $u_1$ is isometrically related to a KKPS-type field follows from the discussion following Equation \eqref{nomizudef} and by noticing that the isometry $\Psi$ indicated there preserves the fields in the standard frame \eqref{stdframe}. 

\medskip
\noindent $(iii)$ We directly compute $u_a(\abs{u_a}^2)=16 a (a^2 - 1) \sin 2 s \sin 4 s \sin 2 (\phi_1 - \phi_2)$. Therefore, if $a \neq 0,1$, the field $u_a$ is not localizable. Since all KKPS solutions are necessarily  localizable, by Lemma \ref{conservnorm} the conclusion follows.
\end{proof}

The second family of new solutions is  of the form $u = f_1 K_1 + f_2 K_2$ briefly considered in the Introduction.
\begin{pr} \label{solab}
Consider the following family vector fields on $\Ss^3$, indexed by two parameters $a_1$, $a_2 \in \RR$, at least one of them being non-zero ($a_1^2+a_2^2\neq 0$):
\begin{equation} \label{kkps2}
u= \sin 2s \big( a_1 \sin(\phi_1 + \phi_2) X_1 +
a_2 \cos(\phi_1 + \phi_2) X_2 \big).
\end{equation}
The following statements hold true.

\noindent $(i)$ For any $a_1, a_2$, the vector field $u$ is a steady Euler solution on $\Ss^3$ with Bernoulli function 
\begin{equation} \label{bernkkps2}
b=\tfrac{1}{4}\big( a_1 a_2 \cos 4s + 
(a_1^2 + a_2^2 - (a_1^2 - a_2^2) \cos 2(\phi_1 + \phi_2)) \sin^2 2s\big).
\end{equation}

\medskip
\noindent $(ii)$ If $a_1=a_2$, then $u$ is a $\cu$ eigenvector (strong Beltrami field) for the  eigenvalue  $\mu = 4$. Conversely, if $u$ is a Beltrami field, then $a_1=a_2$.

\medskip
\noindent $(iii)$ If $a_1 = -a_2$, then $u$ is a KKPS steady Euler solution, while if $a_1 \neq -a_2$, and  $a_1 a_2  \neq 0$, then there is no isometry relating $u$ to any member of the KKPS family \eqref{kkps}.
\end{pr}

\begin{proof} 
\noindent $(i)$ By direct computation we find $\nabla_u u = a_1 a_2 \sin 4s \, \partial_s$, so $u$ satisfies the first equation in \eqref{eul} with 
$p = \frac{1}{4}a_1 a_2 \cos 4s$. To see also that $\di u =0$, remind that $X_1, X_2$ are Killing vector fields and check that $X_1(\sin 2s \sin(\phi_1 + \phi_2))=0$ and $X_2(\sin 2s \cos(\phi_1 + \phi_2))=0$. The Bernoulli function $b$ is easily computed using the pressure $p$ identified above (up to an additive constant) and the norm of $u$.

\medskip
\noindent $(ii)$ If $a_1=a_2$, then the Bernoulli function $b$ is identically constant and \eqref{eulb} shows us that $u$ is a Beltrami field. It remains to verify that the proportionality factor is constant. One option is to check that $u$ verifies the equations \cite[(2.7)]{PSS}, showing that it is actually an eigenvector for the eigenvalue $\mu=4$ (tangent to the contact distribution $\ker \eta$). Otherwise, notice that the coefficients in the standard basis $f_i=\langle u, X_i\rangle$ are harmonic homogeneous polynomials of degree 2 on $\RR^4$ ($f_1=2a_1(y_1x_2 + x_1y_2)$, $f_2=2a_2(x_1x_2 - y_1y_2)$), so eigenfunctions of the Laplacian on $\Ss^3$ for the eigenvalue $\lambda=8$. It is known that the components in the standard basis of curl $\mu$-eigenvectors on the round $\Ss^3$ are Laplacian eigenfunctions with eigenvalue $\mu(\mu-2)$ (see e.g. \cite{PSS}).

For the converse, compute first $X_1(b)= -a_2 (a_1 - a_2)\sin 4s \cos(\phi_1 + \phi_2)$ and $X_2(b)= a_1 (a_1 - a_2)\sin 4s \sin(\phi_1 + \phi_2)$. If $u$ is a Beltrami field, then $b$ is a constant (this has to happen globally because we are in the analytic setting), and, in particular, $X_1(b)$ and $X_2(b)$ computed above must vanish everywhere, so $a_1 = a_2$. 

\medskip
\noindent $(iii)$ As $\xi'= \cos 2s \, \xi + \sin 2s \sin(\phi_1 + \phi_2) X_1 - \sin 2s \cos(\phi_1 + \phi_2) X_2$, if $a_1=-a_2$ we have $u=a_1 (-\cos 2s \, \xi + \xi')$ which is of the KKPS form \eqref{kkps}.

\noindent Let us check now the localizability property of $u$. By direct computation we obtain
$u(\abs{u}^2)= a_1 a_2 (a_1 + a_2) \sin 2 s  \sin 4s  \sin  2(\phi_1 + \phi_2)$ so if $a_1 \neq  -a_2$ and  $a_1 a_2  \neq 0$, the field is not localizable. But KKPS solutions are all localizable, so Lemma \ref{conservnorm} yields the conclusion. 
\end{proof}

\begin{re}[Mirror solutions]
The vector fields from the following  family also verify the steady Euler equations on $\Ss^3$:
$$
u= a_1\sin 2s \sin(\phi_1 - \phi_2) X_1^\prime + a_2 \sin 2s \cos(\phi_1 - \phi_2) X_2^\prime.
$$
They can be isometrically related to \eqref{kkps2} (so they do not represent new solutions) through an isometry that changes $\phi_2$ into $-\phi_2$, transforming the standard frame \eqref{stdframe} into the "mirror frame" \eqref{otherhopf}. We call the solutions obtained in this way, \emph{mirror solutions}. Notice that the KKPS  solutions are their own mirrors in the sense that the mirror of a solution \eqref{kkps} has the same form. See section \ref{twins} for more mirror solutions.
\end{re}

Now we are ready to prove our non-isolation result Theorem \ref{maint} (see Introduction).

\begin{proof}[Proof of Theorem \ref{maint}]
Let us consider the smooth 1-parameter family $u_a$ defined by \eqref{kkps2} with $a_1=1$ and $a_2=a-1$, for $a\geq 0$. We have $u_a - u_0 = a \sin 2s \cos(\phi_1 -\phi_2) X_2$ and it is immediate to see that there exists a positive constant $C$ such that:
\begin{equation} \label{estim}
 \norm{u_a - u_0}_{C^k}:=\sup \abs{u_a - u_0} + \sum_{i=1}^k \sup  \abs{\nabla^i (u_a - u_0)}  < C a,
\end{equation}
where $\abs{\cdot}$ is the norm of vector fields induced by the metric and $\nabla^i$ is the $i^{th}$ covariant derivative. This shows us that $u_a$ is as close as we wish to the KKPS solution $u_0$ with respect to the distance given by the $C^k$ norm ($k\in \NN$ arbitrarily fixed). At the same time, according to Proposition \ref{solab} $(iii)$, $u_a$, $a>0$ is not (isometric to) a KKPS solution, that completes the proof of the first statement.

Let us consider now the family $u_a$, $a\geq 0$ in \eqref{solnomizu} that bifurcates from the KKPS solution $u_0$. In this case $u_a - u_0 = [F(\psi_a)-F(\psi_0)] \xi + \frac{a}{2}\phi (\gr \, (\sin^2(2 s) \cos 2 (\phi_1 - \phi_2)))$. In particular, if $F$ is the linear function $F(x)=2(3x-2)$ for which $u_a$ is a strong Beltrami field (eigenvector of $\cu$ operator) cf. Proposition \ref{solnomi}$(i)$, then the estimate \eqref{estim} holds true. So $u_{a>0}$ may be arbitrarily $C^k$-close to the KKPS curl eigenvector $u_0$, but they are not isometrically related, cf. Proposition \ref{solnomi}$(iii)$, thus proving the second statement.

In the case when $F$ is non-linear, the conclusion remains true, but the estimate \eqref{estim} is less straightforward to obtain (it is enough to estimate the $C^k$-norm of $[F(\psi_a)-F(\psi_0)] \xi$ that reduces to show that the absolute value of the partial derivatives up to $k$ of $F(\psi_a)-F(\psi_0)$ are upper bounded by $Ca$; the latter follows from the mean value theorem).
\end{proof}

\begin{re}
It is worth to compare this result with \cite[Theorem 3.4]{kkpsnew} stating that, on $\Ss^3$, if there exists a steady Euler solution $C^k$-close to a \emph{stationary nondegenerate shear Euler flow} (which is a particular type of KKPS solution), then their Bernoulli functions are similar (in a precise sense), up to a small deformation. We mention that the KKPS solutions $u_{a=0}$ that we deform are not nondegenerate shear Euler flows, since the critical set of the Bernoulli function is larger that the Hopf link (see \cite{kkpsnew} for definitions).
\end{re}

Let us finally notice that our Theorem \ref{maint} establishes the non-isolation of specific KKPS solutions; it would be interesting to know a larger class of non-isolated KKPS solutions or to find out whether there exist isolated KKPS solutions.

\section{Appendix}
\subsection{Adapted orthonormal frames on Sasakian 3-manifolds} For the convenience of the reader we reproduce here some useful properties of adapted orthonormal frames $\{\xi, X_1, X_2=-\phi X_1\}$ on a Sasakian 3-manifold $M$, cf. \cite{PSS}. 

The commutation relations that define the \emph{structure functions} $C_i$, $i=0,1,2$ are:
\begin{equation} \label{struct1}
\begin{split}
& [\xi, X_1]=-(C_0+1)X_2, \ \  [X_1, X_2]=-2\xi+C_1 X_1 + C_2 X_2, 
\ \ [X_2, \xi]=-(C_0+1)X_1.
\end{split}
\end{equation}

The connection coefficients deduced from \eqref{struct1} read:
\begin{equation}\label{conne}
\left\{
\begin{array}{cccc}
\nabla_{\xi}\xi=0, &
\nabla_{X_1}\xi=X_2 , &
\nabla_{X_2}\xi=-X_1  \\[3mm]
\nabla_{\xi}X_1=-C_0 \, X_2 , &
\nabla_{X_1}X_1= -C_1 X_2 , &
\nabla_{X_2}X_1=\xi - C_2 \, X_2 \\[3mm]
\nabla_{\xi}X_2= C_0 \, X_1, &
\nabla_{X_1}X_2=-\xi + C_1 \, X_1, &
\nabla_{X_2}X_2 = C_2 \, X_1 \\[3mm]
\end{array}
\right.
\end{equation}

Let $X = f\xi + f_1 X_1 + f_2 X_2$ be an arbitrary vector field on $M$. Then
\begin{equation} \label{curl123}
\begin{split}
\cu X = & \big(X_1(f_2) - X_2(f_1)- C_1f_1 -C_2 f_2 +2f \big)\xi +\\
& \big(-\xi(f_2) + X_2(f) + (C_0+1) f_1 \big)X_1 + \\
& \big(\xi(f_1) - X_1(f) + (C_0+1) f_2 \big)X_2.
\end{split}
\end{equation}
and 
\begin{equation} \label{div123}
\di X = \xi(f)+ X_1(f_1) + X_2(f_2) - C_2 f_1 + C_1 f_2.
\end{equation}

\subsection{The 3-sphere geometry basics}

The sphere $\Ss^3$ is seen a the set of points $(z_1, z_2) \in \CC^2$ with $\abs{z_1}^2 + \abs{z_2}^2 =1$. Denoting $z_j= x_j + \ii y_j$, at each point $(x_1, y_1, x_2, y_2) \in \Ss^3$ we have the orthonormal (global) frame of Killing vector fields, which, at the same time, form a $L^2$-basis in the space of eigenvectors of $\cu$ operator for the lowest positive eigenvalue $\mu_1=2$:
\begin{equation} \label{stdframe}
\begin{split}
\xi     & = - y_1 \partial_{x_1} + x_1 \partial_{y_1} - y_2 \partial_{x_2} + x_2 \partial_{y_2} ,\\
X_1 & =-x_2 \partial_{x_1} + y_2 \partial_{y_1} + x_1 \partial_{x_2} - y_1 \partial_{y_2} ,\\
X_2 & =-y_2 \partial_{x_1} - x_2 \partial_{y_1} + y_1 \partial_{x_2} + x_1 \partial_{y_2}.
\end{split}
\end{equation}
The structure functions in \eqref{struct1} are identically constant: $C_0=1$, $C_1 = C_2=0$.
%on which Levi-Civita connection act as follows:
%\begin{equation}\label{conne}
%\left\{
%\begin{array}{cccc}
%\nabla_{\xi}\xi=0, &
%\nabla_{X_1}\xi=X_2 , &
%\nabla_{X_2}\xi=-X_1  \\[3mm]
%\nabla_{\xi}X_1= - X_2 , &
%\nabla_{X_1}X_1= 0 , &
%\nabla_{X_2}X_1=\xi  \\[3mm]
%\nabla_{\xi}X_2=  X_1, &
%\nabla_{X_1}X_2= - \xi , &
%\nabla_{X_2}X_2 = 0 \\[3mm]
%\end{array}
%\right.
%\end{equation}

\medskip

\noindent It is often useful to work in the \textit{Hopf coordinates}: $(x_1, y_1, x_2, y_2)=(\cos s \, e^{i\phi_1}, \sin s \, e^{i\phi_2})$, $s \in [0, \pi/2]$, $\phi_i \in [0, 2\pi)$. The round metric reads $g=\dif s ^2 + \cos^2 s \dif \phi_1^2 + \sin^2 s \dif \phi_2^2$ and standard orthonormal frame above becomes
\begin{equation} \label{hopf}
\begin{split}
\xi &= \partial_{\phi_1} + \partial_{\phi_2}, \\
X_1 &= \cos(\phi_1 + \phi_2)\partial_s +\sin(\phi_1 + \phi_2)(\tan s \, \partial_{\phi_1} - \cot s \, \partial_{\phi_2}), \\
X_2 &= \sin(\phi_1 + \phi_2)\partial_s -\cos(\phi_1 + \phi_2)(\tan s \, \partial_{\phi_1} - \cot s \, \partial_{\phi_2}) .
\end{split}
\end{equation}

\medskip

\noindent The eigenspace of $\cu$ operator associated to the first negative eigenvalue $\mu_{-1} = -2$ is spanned by the following three Killing vector fields (obtained from \eqref{hopf} by using the orientation reversing isometry $(x_1, y_1, x_2, y_2)  \mapsto (x_1, y_1, x_2, - y_2)$ of $\Ss^3$):
\begin{equation} \label{otherhopf}
\begin{split}
\xi^{\prime} &= \partial_{\phi_1} - \partial_{\phi_2}, \\
X_1^{\prime} &= \cos(\phi_1 - \phi_2)\partial_s +\sin(\phi_1 - \phi_2)(\tan s \partial_{\phi_1} + \cot s \partial_{\phi_2}), \\
X_2^{\prime} &= \sin(\phi_1 - \phi_2)\partial_s -\cos(\phi_1 - \phi_2)(\tan s \partial_{\phi_1} + \cot s \partial_{\phi_2}) .
\end{split}
\end{equation}

\subsection{Twin and mirror solutions}\label{twins}
The same solution can appear disguised in various forms, that are isometrically related. We provide here a list of such "twins" of the solutions discussed in this paper.

\noindent   The following vector fields (the second being the "mirror" of the first one)
\begin{equation} \label{kkps1}
\begin{split}
u&=F(\sin 2s \sin(\phi_1 + \phi_2)) X_1 + G(\sin 2s \sin(\phi_1 + \phi_2)) \xi^\prime \\
u&=F(\sin 2s \sin(\phi_1 - \phi_2)) X_1^\prime + G(\sin 2s \sin(\phi_1 - \phi_2)) \xi
\end{split}
\end{equation}
are steady solutions of the Euler equation for any choice of functions $F$ and $G$. Remark that $\sin 2s \sin(\phi_1 + \phi_2)$ is a first integral for both fields $X_1$ and $\xi^\prime$. The isometry defined by  $x_1^\prime = \frac{1}{\sqrt{2}} (x_1 - y_2), y_1^\prime = \frac{1}{\sqrt{2}} (y_1 - x_2), 
 x_2^\prime = \frac{1}{\sqrt{2}} (x_2 + y_1), y_2^\prime =  \frac{1}{\sqrt{2}} (x_1 + y_2)$ pulls $\xi$ back\footnote{Here by pull-back of a vector field $X$ through some mapping $\varphi$ we mean $(\varphi^* X^\flat)^\sharp$ and the pull-back of a function $f$ on the codomain is $f\circ \varphi$ defined on the domain of $\varphi$.} to $-X_1$, the prime integral $x_1^2 + y_1^2 - x_2^2 - y_2^2$ back to $-2 (x_1 y_2 + x_2 y_1)$, and it preserves $\xi'$. Therefore the solutions \eqref{kkps1} are isometrically equivalent to the KKPS solutions.

\medskip

\noindent Analogously we have the solutions:
\begin{equation} \label{kkps1b}
\begin{split}
u&=F(\sin 2s \cos(\phi_1 + \phi_2)) X_2 + G(\sin 2s \cos(\phi_1 + \phi_2)) \xi^\prime \\
u&=F(\sin 2s \cos(\phi_1 - \phi_2)) X_2^\prime + G(\sin 2s \cos(\phi_1 - \phi_2)) \xi
\end{split}
\end{equation}

\bigskip

\noindent Another solutions akin to \eqref{kkps2} are
\begin{equation} \label{kkps2alter}
\begin{split}
u & = a\cos 2s \, \xi + b \sin 2s \sin(\phi_1 + \phi_2) X_1, \\
u & = a\cos 2s \, \xi + b \sin 2s \cos(\phi_1 + \phi_2) X_2.
\end{split}
\end{equation}

\bigskip

\noindent A solution based on $X_1$ and $X_2^\prime$ (and its mirror) is:
\begin{small}
\begin{equation} \label{kkps3}
\begin{split}
u & = F(\cos^2 s\, \sin(2\phi_1)-\sin^2 s\, \sin(2\phi_2)) X_1 +
G(\cos^2 s\, \sin(2\phi_1)-\sin^2 s\, \sin(2\phi_2)) X_2^\prime \\
u & = F(\cos^2 s\, \sin(2\phi_1)+\sin^2 s\, \sin(2\phi_2)) X_1^\prime +
G(\cos^2 s\, \sin(2\phi_1)+\sin^2 s\, \sin(2\phi_2)) X_2.
\end{split}
\end{equation}
\end{small}

\noindent Finally, we have the solution based on $X_1$ and $X_1^\prime$ and on $X_2$ and $X_2^\prime$ (which are their own mirrors):
\begin{small}
\begin{equation} \label{kkps3}
\begin{split}
u & = F(\cos^2 s\, \cos(2\phi_1)+\sin^2 s\, \cos(2\phi_2)) X_1 +
G(\cos^2 s\, \cos(2\phi_1)+\sin^2 s\, \cos(2\phi_2)) X_1^\prime \\
u & = F(\cos^2 s\, \cos(2\phi_1)-\sin^2 s\, \cos(2\phi_2)) X_2 +
G(\cos^2 s\, \cos(2\phi_1)-\sin^2 s\, \cos(2\phi_2)) X_2^\prime.
\end{split}
\end{equation}
\end{small}

%\noindent \textbf{Acknowledgements}. I thank D. Peralta-Salas for suggesting me this problem and for many helpful discussions.

\end{document}